\documentclass{amsart}
\usepackage{latexsym}
\def\qed{\hfill $\Box$}
\usepackage{amsfonts,amssymb,amsmath}
\usepackage{amsthm}
\usepackage{graphicx,color}
\usepackage{amsmath} 
\usepackage{multirow,bigdelim}
\usepackage{cleveref}
\newtheorem{theorem}{Theorem}	
\crefname{theorem}{Theorem}{Theorems}
\newtheorem{lemma}{Lemma}
\newtheorem{corollary}{Corollary}		
\newtheorem{proposition}{Proposition}	
\crefname{proposition}{Proposition}{Propositions}

\newtheorem{definition}{Definition}

\newtheorem{example}{Example}
\newtheorem{remark}{Remark}

\crefname{section}{Section}{Sections}
\crefname{theorem}{Theorem}{Theorems}
\crefname{lemma}{Lemma}{Lemmas}
\crefname{corollary}{Corollary}	{Corollaries}			
\crefname{proposition}{Proposition}{Propositions}	
\crefname{claim}{Claim}{Claims}
\crefname{conjecture}{Conjecture}{Conjectures}			
\crefname{definition}{Definition}{Definitions}
\crefname{problem}{Problem}{Problems}
\crefname{example}{Example}{Examples}
\crefname{remark}{Remark}{Remarks}
\crefname{figure}{Figure}{Figures}

\newfont{\bg}{cmr9 scaled\magstep2}
\newcommand{\bigzerol}{\smash{\lower1.0ex\hbox{\bg 0}}}

\newcommand{\R}{\mathbb{R}}

\makeatletter
\newcommand{\tpitchfork}{%
  \vbox{
    \baselineskip\z@skip
    \lineskip-.52ex
    \lineskiplimit\maxdimen
    \m@th
    \ialign{##\crcr\hidewidth\smash{$-$}\hidewidth\crcr$\pitchfork$\crcr}
  }%
}

\makeatother
\title[Generic distance-squared mappings on plane curves]
{
Generic distance-squared mappings 
\\
on plane curves
}
\author{Shunsuke Ichiki
}
\thanks{Research Fellow DC1 of Japan Society for the Promotion of Science}
\address{
Graduate School of Environment and Information Sciences,  
Yokohama National University, 
Yokohama 240-8501, Japan}
\email{ichiki-shunsuke-jb@ynu.jp}

\begin{document}
\date{}
\begin{abstract}
A distance-squared function is one of the most significant functions 
in the application of singularity theory to differential geometry.
Moreover, distance-squared mappings are naturally extended mappings of distance-squared functions, wherein each component is a distance-squared function. 
In this paper, 
compositions of a given plane curve and generic distance-squared mappings 
on the plane into the plane 
are investigated from the viewpoint of stability. 
\end{abstract}
\maketitle
\noindent
\section{Introduction}\label{sec:intro}
Throughout this paper, let $\ell$ and $n$ stand for positive integers. 
In this paper, unless otherwise stated, all manifolds and mappings belong to class $C^{\infty}$ and all manifolds are without boundary. 
Let $p$ be a given point in $\mathbb{R}^n$. 
The mapping $d_p:\mathbb{R}^n\to \mathbb{R}$ defined by 
\[
d_p(x)=\sum_{i=1}^n(x_i-p_i)^2
\]
is called a distance-squared function, 
where $x=(x_1,\ldots ,x_n)$ and $p=(p_1,\ldots ,p_n)$. 
In \cite{D}, the following notion is investigated. 
\begin{definition}\label{def:D}
{\rm Let $p_1,\ldots,p_\ell$ be $\ell$ given points in $\mathbb{R}^n$. 
Set $p=(p_1,\ldots ,p_\ell)\in (\mathbb{R}^n)^\ell$. 
The mapping $D_p : \mathbb{R}^n \rightarrow \mathbb{R}^\ell$ 
defined by 
\[
D_p=(d_{p_1},\ldots , d_{p_\ell})
\]
is called a {\it distance-squared mapping}}.   
\end{definition}
We have the following important motivation for investigating distance-squared mappings. 
Height functions and distance-squared functions have been investigated 
in detail so far, and they are useful tools  
in the applications of singularity theory to differential geometry (see \cite{CS}). 
A mapping in which each component is a height function is nothing but a projection. 
Projections as well as height functions or distance-squared functions have been 
investigated so far. 
For example, in \cite{GP} (resp., \cite{brucekirk}), 
compositions of generic projections and embeddings (resp., stable mappings) 
are investigated from the viewpoint of stability. 
On the other hand, 
a mapping in which each component 
is a distance-squared function is a distance-squared mapping. 
Therefore, it is natural to investigate distance-squared mappings 
as well as projections. 

A mapping $f : \mathbb{R}^n\rightarrow \mathbb{R}^\ell$ is said to be {\it $\mathcal{A}$-equivalent} to a mapping $g : \mathbb{R}^n\rightarrow \mathbb{R}^\ell$ 
if there exist diffeomorphisms $\varphi : \mathbb{R}^n\rightarrow \mathbb{R}^n$ and $\psi : \mathbb{R}^\ell\rightarrow \mathbb{R}^\ell$ 
such that $\psi \circ f\circ \varphi^{-1} =g$. 
$L$ points $p_1,\ldots,p_\ell\in  \R^n$ $(1\leq \ell\leq n+1)$ are said to be {\it in general position}  if $\ell=1$ or 
$\overrightarrow{p_1p_2},\ldots,\overrightarrow{p_1p_\ell}$ $(2\leq \ell\leq n+1)$ are linearly independent.  

In \cite{D}, 
a characterization of distance-squared mappings is given as follows:  
\begin{proposition}[\cite{D}]\label{property}
\begin{enumerate}
\item [$(1)$]
Let $\ell$,$n$ be integers such that $2\leq \ell\leq n$, and let $p_1,\ldots,p_\ell\in \mathbb{R}^n$ be in general position. Then, $D_p : \mathbb{R}^n\rightarrow \mathbb{R}^\ell$ is $\mathcal{A}$-equivalent 
to $(x_1,\ldots,x_n)\mapsto (x_1,\ldots,x_{\ell-1}, x^2_\ell+\cdots +x^2_n)$.
\item [$(2)$]
Let $\ell$,$n$ be integers such that $1\leq n<\ell $, and let $p_1,\ldots,p_{n+1}\in \mathbb{R}^n$ be in general position. Then, $D_p : \mathbb{R}^n\rightarrow \mathbb{R}^\ell$ is $\mathcal{A}$-equivalent to the inclusion $(x_1,\ldots,x_n)\mapsto (x_1,\ldots,x_n,0,\ldots,0)$.
\end{enumerate}
\end{proposition}


In the following, by $N$, we denote a manifold of dimension 1. 
A mapping $f:N\to \R^2$ is called 
a {\it mapping with normal crossings} if the mapping $f$ satisfies the following conditions.  
\begin{enumerate}
\item
For any $y\in \R^2$, $|f^{-1}(y)|\leq 2$, where $|A|$ is the number of its elements of the set $A$. 
\item 
For any two different points $q_1, q_2\in N$ 
satisfying $f(q_1)=f(q_2)$, we have 
$
\dim \left( df_{q_1}(T_{q_1}N)+df_{q_2}(T_{q_2}N)\right)=2$.
\end{enumerate}


From Corollary 8 in \cite{C}, we have the following.  
\begin{proposition}[\cite{C}]\label{semimain}
Let $\gamma: N\to \R^2$ be an injective immersion, where $N$ 
is a manifold of dimension $1$.  
 Then, the following set 
\begin{eqnarray*}
\left \{ p\in \R^2\times \R^2 \mid D_{p}\circ \gamma : N\to \R^2 {\mbox { is an immersion with normal crossings}} \right \}
\end{eqnarray*}
is dense in $\R^2 \times \R^2$. 
\end{proposition}

On the other hand, the purpose of this paper is to 
investigate whether the following set 
\begin{eqnarray*}
\left \{ p\in \gamma(N) \times \gamma(N) \mid D_{p}\circ \gamma : N\to \R^2 {\mbox { is an immersion with normal crossings}} \right \}
\end{eqnarray*}
is dense in $\gamma(N) \times \gamma(N)$ or not. 
Here, note that $O$ is an open set of $\gamma(N)\times \gamma(N)$ 
if there exists an open set $O'$ of $\R^2\times \R^2$ satisfying 
$O=O'\cap(\gamma(N)\times \gamma(N))$. 

Let $\gamma: N\to \R^2$ be an immersion. 
We say that $\kappa : U\to \R$ is called the {\it curvature} of $\gamma$ 
on a coordinate neighborhood $(U,t)$ if 
\begin{eqnarray*}
\kappa(t)&=&\frac
{
 \det 
\left(
\begin{array}{cc}
\displaystyle{\frac{d \gamma_2}{dt}(t)} & 
-\displaystyle{\frac{d \gamma_1}{dt}(t)} 
\\
 [4.5mm] 
\displaystyle{\frac{d ^2\gamma_2}{dt^2}(t)} & 
-\displaystyle{\frac{d ^2\gamma_1}{dt^2}(t)} 
\end{array}
\right)
}
{
\left(\displaystyle{\frac{d \gamma_1}{dt}(t)}^2
+
\displaystyle{\frac{d \gamma_2}{dt}(t)}^2
\right)^{\frac{3}{2}}
}, 
\end{eqnarray*}
where $\gamma=(\gamma_1,\gamma_2)$. 
Note that it dose not depend on the choice of a coordinate 
neighborhood whether $\kappa(q)=0$ for a given point $q\in N$ or not.  
\begin{definition}
{\rm 
Let $N$ be a manifold of dimension 1. 
We say that an immersion $\gamma: N\to \R^2$ {\it satisfies $(\ast)$} 
if for any non-empty open set $U$ of $N$, 
there exists a point $q\in U$ satisfying $\kappa (q)\not =0$, 
where $\kappa$ 
is the curvature of $\gamma$ on a coordinate neighborhood around $q$. 
}
\end{definition}
The main result in this paper is the following. 
\begin{theorem}\label{main}
Let $\gamma: N\to \R^2$ be an injective immersion 
satisfying $(\ast)$, where $N$ is a manifold of dimension $1$. 
Then, the following set 
\begin{eqnarray*}
\left \{ p\in \gamma(N) \times \gamma(N) \mid D_{p}\circ \gamma : N\to \R^2 {\mbox { is an immersion with normal crossings}} \right \}
\end{eqnarray*}
is dense in $\gamma(N) \times \gamma(N)$. 
\end{theorem}
If we drop the hypothesis $(\ast)$ in \cref{main}, 
then the conclusion of \cref{main} does not necessarily hold 
(see \cref{sec:drop}).  

In \cref{main}, if the mapping $D_{p}\circ \gamma : N\to \R^2$ 
is proper, then the immersion with normal crossings $D_{p}\circ \gamma : N\to \R^2$ 
is necessarily stable (see \cite{GG}, p. 86). 
Thus, from \cref{main}, we get the following. 
\begin{corollary}\label{maincoro}
Let $N$ be a compact manifold of dimension $1$. 
Let $\gamma: N\to \R^2$ be an embedding satisfying $(\ast)$. 
 Then, the following set 
\begin{eqnarray*}
\left \{ p\in \gamma(N) \times \gamma(N) \mid D_{p}\circ \gamma : N\to \R^2 {\mbox { is stable}} \right \}
\end{eqnarray*}
is dense in $\gamma(N) \times \gamma(N)$. 
\end{corollary}

%

\par 
\bigskip 
In \cref{sec:drop}, two examples that the conclusion of 
\cref{main} does not necessarily hold 
without the hypothesis $(\ast)$ in the theorem are given. 
In \cref{sec:mainlem}, some assertions for the proof of \cref{main} 
are prepared. 
\cref{sec:main} is devoted to the proof of \cref{main}.

\section{Dropping the hypothesis $(\ast)$ in \cref{main}}\label{sec:drop}
In this section, two examples that the conclusion of 
\cref{main} does not necessarily hold 
without the hypothesis $(\ast)$ in the theorem are given 
(see \cref{ex:1,ex:2}). 
 
Firstly, we prepare the following proposition, which is used in \cref{ex:1}.  
\begin{proposition}\label{drop}
Let $\gamma: N\to \R^2$ be an immersion, 
where $N$ is a manifold of dimension $1$. 
Let $p_1$, $p_2$ be two points of $\R^2$. 
Then, a point $q\in N$ is a singular point of the mapping 
$D_{p}\circ \gamma : N\to \R^2$ $(p=(p_1,p_2))$ if and only if 
\begin{eqnarray*}
\overrightarrow{p_1\gamma(q)}\cdot \frac{d \gamma}{dt}(q)=0 
{\mbox {\  and \ }} 
\overrightarrow{p_2\gamma(q)}\cdot \frac{d \gamma}{dt}(q)=0, 
\end{eqnarray*}
where $t$ is a local coordinate around the point $q$ and 
$\cdot $ stands the inner product of $\R^2$. 
\end{proposition}
\begin{proof}
Let $q$ be a point of $N$.  
The composition of $\gamma: N\to \R^2$ and $D_{p} : \R^2 \to \R^2$ is given 
as follows:   
\begin{eqnarray*}
D_{p}\circ \gamma (q)=\left( (\gamma_{1}{(q)}-p_{11})^2+(\gamma_{2}(q)-p_{12})^2, 
(\gamma_{1}(q)-p_{21})^2+(\gamma_{2}(q)-p_{22})^2 \right),
\end{eqnarray*}
where  
$p_1=(p_{11}, p_{12})$, $p_2=(p_{21}, p_{22})$ and 
$\gamma =(\gamma_1,\gamma_2)$. 

Then, we have 
\begin{eqnarray*}
\frac{dD_{p}\circ \gamma}{dt}(q)
&=&2\left ((\gamma_{1}{(q)}-p_{11})\frac{d\gamma_1}{dt}(q)+
(\gamma_{2}(q)-p_{12})\frac{d\gamma_2}{dt}(q), \right. 
\\
&&\left.
(\gamma_{1}(q)-p_{21})\frac{d\gamma_1}{dt}(q)+
(\gamma_{2}(q)-p_{22})\frac{d\gamma_2}{dt}(q)\right)
\\
&=&2\left(\overrightarrow{p_1\gamma(q)}\cdot \frac{d \gamma}{dt}(q), 
\overrightarrow{p_2\gamma(q)}\cdot \frac{d \gamma}{dt}(q)
\right), 
\end{eqnarray*}
where $t$ is a local coordinate around the point $q$. 
Hence, a point $q$ is a singular point of the mapping $D_{p}\circ \gamma$ 
if and only if 
\begin{eqnarray*}
\left(\overrightarrow{p_1\gamma(q)}\cdot \frac{d \gamma}{dt}(q), 
\overrightarrow{p_2\gamma(q)}\cdot \frac{d \gamma}{dt}(q)
\right)=(0,0).
\end{eqnarray*}
\end{proof} 

\begin{example}\label{ex:1}
{\rm 
In this example, we use \cref{drop}. 
Let $\gamma: S^1\to \R^2$ be an embedding 
such that $\gamma(S^1)$ is given by \cref{fig:ex1}. 
Here, note that 
there exists an open set $U$ of $N$ satisfying 
for any $q\in U$, $\kappa(q)=0$ (see $\gamma(U)$ in \cref{fig:ex1}). 
Namely, $\gamma$ does not satisfy $(\ast)$. 

Let $p=(p_1,p_2)\in \gamma (U)\times \gamma(U)$ be any point.  
Then, we will show that  
the mapping $D_{p}\circ \gamma$ is not an immersion.  
From \cref{fig:ex1}, it is clearly seen that 
\begin{eqnarray*}
\overrightarrow{p_1\gamma(q')}\cdot \frac{d \gamma}{dt}(q')=0 
{\mbox {\  and \ }} 
\overrightarrow{p_2\gamma(q')}\cdot \frac{d \gamma}{dt}(q')=0, 
\end{eqnarray*}
where $\gamma(q')$ is the point in \cref{fig:ex1} 
and $t$ is a local coordinate around the point $q'$. 
From \cref{drop}, the point $q'$ is a singular point of $D_p\circ\gamma$. 
Namely, for any $p=(p_1,p_2)\in \gamma (U)\times \gamma(U)$, 
the mapping  $D_p\circ\gamma$ is not an immersion. 
Since $\gamma (U)\times \gamma(U)$ is a non-empty open set of 
$\gamma (S^1)\times \gamma(S^1)$, 
the conclusion of \cref{main} does not hold. 
\begin{figure}[h]
\begin{center}
\includegraphics[width=0.8\linewidth]{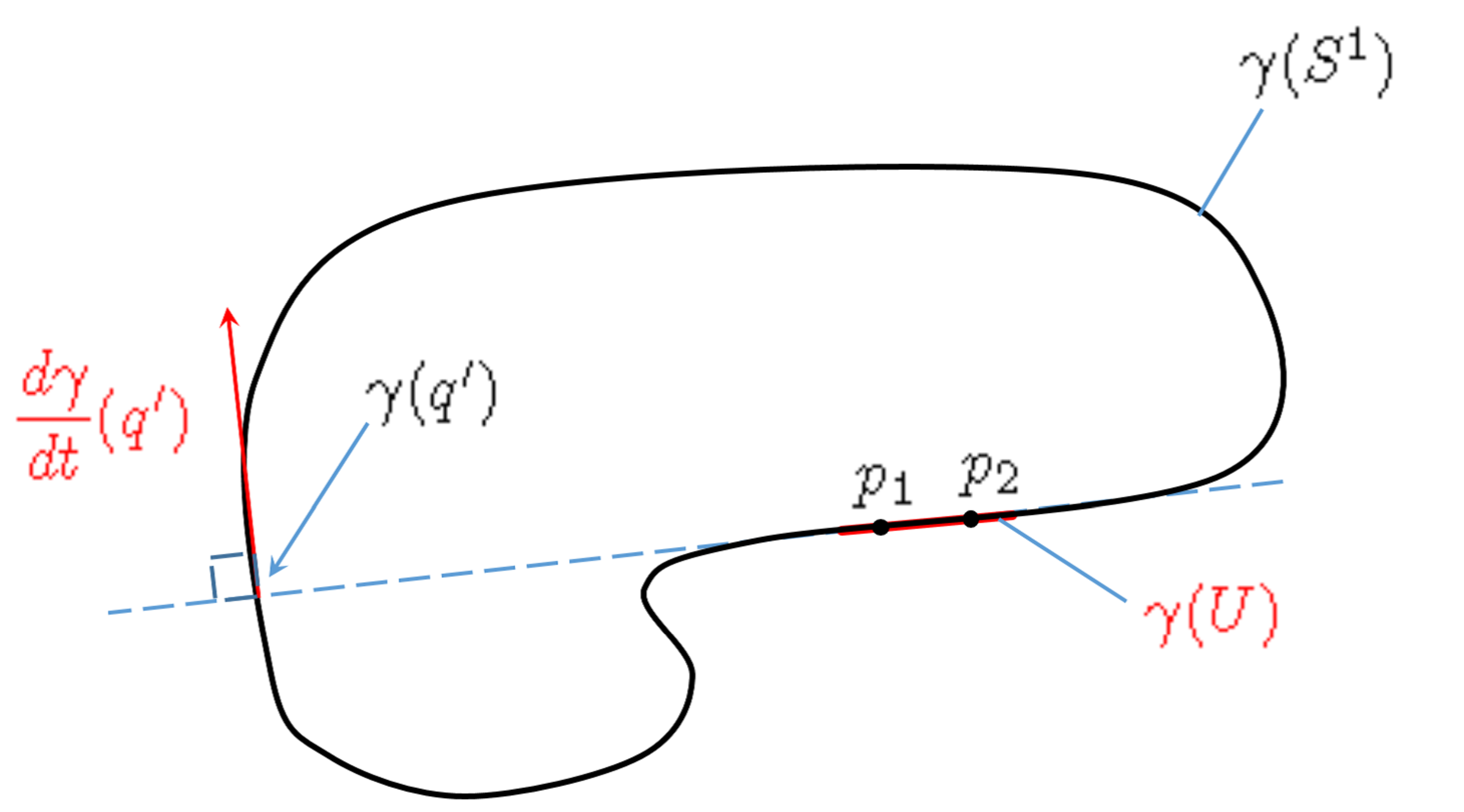} 

\caption{
The figure of Example 1
}
\label{fig:ex1}
\end{center}
\end{figure}
}
\end{example}

\begin{example}\label{ex:2}
{\rm 
Let $I_1$, $I_2$ and $I_3$ be open intervals $(0,1)$, $(1,2)$ and $(2,3)$ of $\R$, 
respectively. 
Let $\gamma : I_1\cup I_2\cup I_3\to \R^2$ be the mapping given by 
\begin{eqnarray*}
\gamma(t) = 
\begin{cases}
(t,-1) & t\in I_1 \\
(t-1,0) & t\in I_2 \\
(t-2,1) & t\in I_3. 
\end{cases}
\end{eqnarray*}
For the image of $\gamma$, see \cref{fig:ex2}. 
Let $p=(p_1, p_2)\in \gamma (I_2)\times \gamma (I_2)$ be any point. 
Then, we will show that $D_{p}\circ \gamma$ is not a mapping with normal crossings. 
By $p_1, p_2\in \gamma (I_2)$, we have 
\begin{eqnarray*}
D_p(x_1,x_2)=\left((x_1-p_{11})^2+x_2^2, (x_1-p_{21})^2+x_2^2\right). 
\end{eqnarray*}
Let $t_0\in I_1$ be any element. Then, 
it follows that $t_0+2\in I_3$ and 
\[
(D_p\circ \gamma)(t_0)=(D_p\circ \gamma)(t_0+2).
\] 
From 
\begin{eqnarray*}
(D_p\circ \gamma)|_{I_1}(t)&=&((t-p_{11})^2+1, (t-p_{21})^2+1), 
\\
(D_p\circ \gamma)|_{I_3}(t)&=&((t-2-p_{11})^2+1, (t-2-p_{21})^2+1), 
\end{eqnarray*}
we get 
\begin{eqnarray*}
d(D_{p}\circ \gamma)_{t_0}&=&2
\left(
\begin{array}{c}
t-p_{11}
\\
t-p_{21}
\end{array}
\right)_{t=t_0}, 
\\
d(D_{p}\circ \gamma)_{t_0+2}&=&2
\left(
\begin{array}{c}
t-2-p_{11}
\\
t-2-p_{21}
\end{array}
\right)_{t=t_0+2}. 
\end{eqnarray*}
Since the rank of the $2\times 2$ matrix 
$(d(D_{p}\circ \gamma)_{t_0}, d(D_{p}\circ \gamma)_{t_0+2})$ is less than two, 
$D_{p}\circ \gamma$ is not a mapping with normal crossings. 
Hence, for any $p=(p_1, p_2)\in \gamma (I_2)\times \gamma (I_2)$, 
$D_{p}\circ \gamma$ is not a mapping with normal crossings. 
\begin{figure}[h]
\begin{center}
\includegraphics[width=0.5\linewidth]{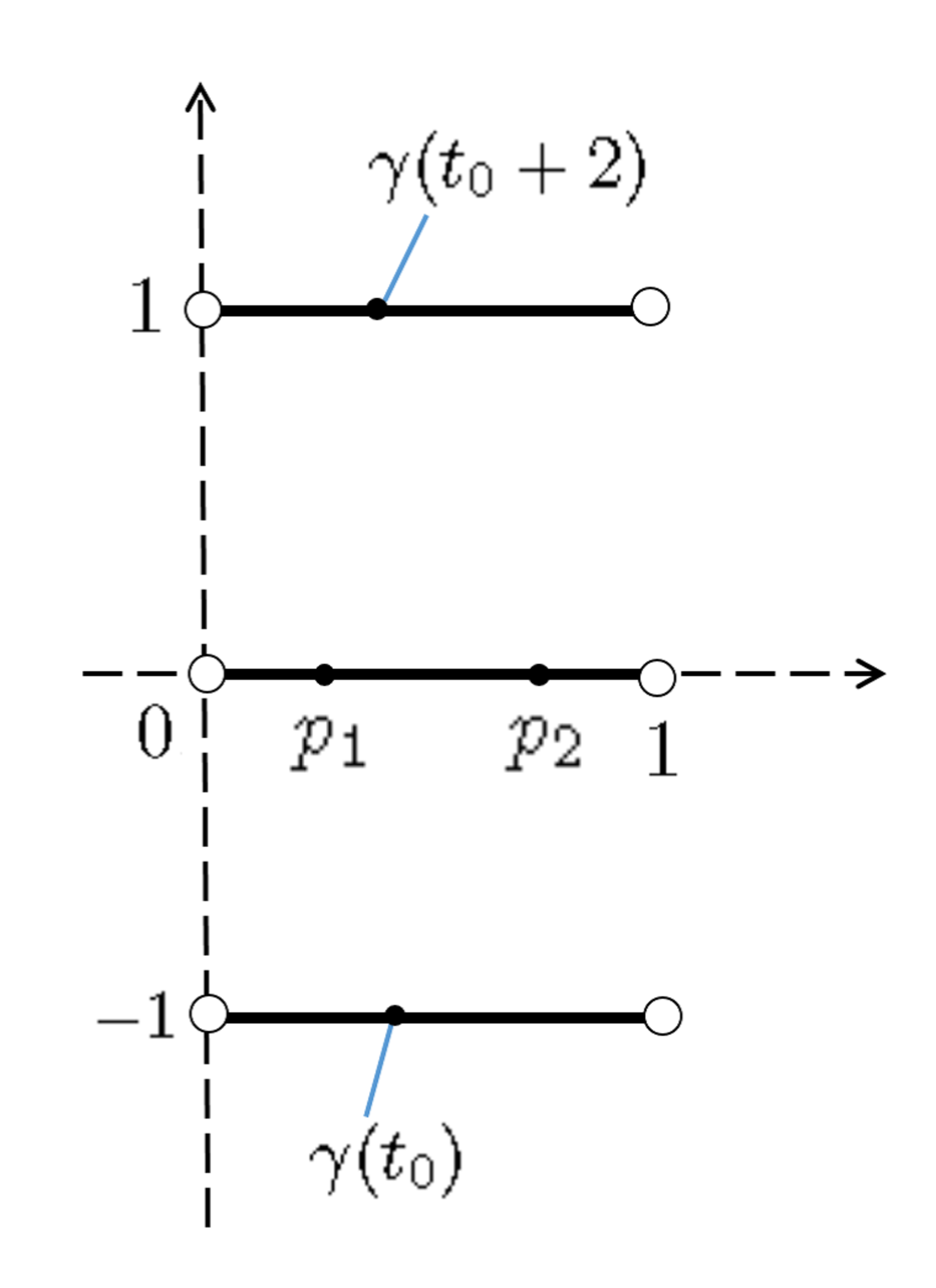} 

\caption{
The figure of Example 2
}
\label{fig:ex2}
\end{center}
\end{figure}
}
\end{example}

\begin{remark}
{\rm 
There is a case that the conclusion of \cref{main} holds 
without the hypothesis $(\ast)$ in the theorem. 
Let $\gamma :\R \to \R^2$ be the mapping defined by 
$\gamma (t)=(t,0)$. 
Set 
\begin{eqnarray*}
A=\left \{ p\in \gamma(\R) \times \gamma(\R) \mid D_{p}\circ \gamma : \R\to \R^2 {\mbox { is an immersion with normal crossings}} \right \}.
\end{eqnarray*}
We will show that $A$ is dense in $\gamma (\R)\times \gamma (\R)$. 
Let $p_1=(p_{11},p_{12}), p_2=(p_{21},p_{22})\in \gamma (\R)$ $(=\R \times \{0\})$ 
be arbitrary points. 
Then, we have 
\begin{eqnarray*}
D_{p}\circ \gamma (t)=\left( (t-p_{11})^2, (t-p_{21})^2 \right),
\end{eqnarray*}
where $p=(p_1,p_2)$. 
It is not hard to see that if $p_{11}\not =p_{21}$, then there exists a 
diffeomorphism $H:\R^2\to \R^2$ such that $H\circ D_{p}\circ \gamma(t)=(t,0)$. 
Namely, if $p_{11}\not =p_{21}$, then $D_{p}\circ \gamma $
is an immersion 
with normal crossings. 
On the other hand, if $p_{11} =p_{21}$, then $D_{p}\circ \gamma$ is not an immersion 
with normal crossings. 
Hence, 
\[A=\{p\in \gamma (\R)\times \gamma (\R) \mid p_{11}\not =p_{21}\}. 
\]
Thus, $A$ is dense in $\gamma (\R)\times \gamma (\R)$. 
}
\end{remark}
\section{Preliminaries for the proof of \cref{main}}\label{sec:mainlem}
For the proof of \cref{main}, we prepare \cref{mainpro} and 
\cref{mainlem}
\begin{proposition}\label{mainpro}
Let $L$ be a straight line of $\R^2$. 
For any $p_1, p_2\in L$ $(p_1\not =p_2)$ and for any 
$\widetilde{p}_1, \widetilde{p}_2\in L$ $(\widetilde{p}_1\not = \widetilde{p}_2)$, 
there exists an affine transformation $H:\R^2 \to \R^2$ such that
\begin{eqnarray*}
H\circ D_p=D_{\widetilde{p}},
\end{eqnarray*}
where $p=(p_1,p_2)$ and $\widetilde{p}=(\widetilde{p}_1,\widetilde{p}_2)$. 
\end{proposition}

\begin{proof}

Set $p_1=(p_{11}, p_{12})$, $p_2=(p_{21}, p_{22})$, 
$\widetilde{p}_1=(\widetilde{p}_{11}, \widetilde{p}_{12})$ and 
$\widetilde{p}_2=(\widetilde{p}_{21}, \widetilde{p}_{22})$. 

Let $H_1:\R^2\to \R^2$ be the linear transformation defined by 
\begin{eqnarray*}
H_1(X_1,X_2)=(X_1,X_1-X_2). 
\end{eqnarray*}
Then, we have 
\begin{eqnarray*}
{ } & { } &
H_1 \circ D_p(x_1, x_2)\\
&=&\left((x_1-p_{11})^2+(x_2-p_{12})^2, 2((p_{21}-p_{11})x_1+(p_{22}-p_{12})x_2)+c_1\right),  
\end{eqnarray*}
where $c_1$ is a constant term. 

Let $H_2:\R^2\to \R^2$ be the affine transformation defined by 
\begin{eqnarray*}
H_2(X_1,X_2)=(X_1,X_2-c_1). 
\end{eqnarray*}
Then, we get 
\begin{eqnarray*}
{ } & { } &
H_2\circ H_1 \circ D_p(x_1, x_2)\\
&=&\left((x_1-p_{11})^2+(x_2-p_{12})^2, 2((p_{21}-p_{11})x_1+(p_{22}-p_{12})x_2)\right). 
\end{eqnarray*}

By $p_1, p_2, \widetilde{p}_1, \widetilde{p}_2 \in L$ and $p_1\not =p_2$, 
there exist $\lambda_1, \lambda_2 \in \R$ satisfying 
\begin{eqnarray}
\widetilde{p}_1&=&p_1+\lambda_1 \overrightarrow{p_1p_2}, \label{eq:vector1}
\\
\widetilde{p}_2&=&p_1+\lambda_2 \overrightarrow{p_1p_2}.  \label{eq:vector2}
\end{eqnarray}
By $\widetilde{p}_{1}\not =\widetilde{p}_2$, 
we get $\lambda_1\not =\lambda_2$. 

Let $H_3:\R^2\to \R^2$ be the linear transformation defined by 
\begin{eqnarray*}
H_3(X_1,X_2)=(X_1-\lambda_1X_2, X_1-\lambda_2 X_2). 
\end{eqnarray*}
Then, we get 
\begin{eqnarray*}
{ } & { } &
H_3\circ H_2\circ H_1 \circ D_p(x_1, x_2)\\
&=&\left(
x_1^2-2(p_{11}+\lambda_1(p_{21}-p_{11}))x_1
+x_2^2-2(p_{12}+\lambda_1(p_{22}-p_{12}))x_2+d_1, 
\right.
\\
&&\left. 
x_1^2-2(p_{11}+\lambda_2(p_{21}-p_{11}))x_1
+x_2^2-2(p_{12}+\lambda_2(p_{22}-p_{12}))x_2+d_2
\right), 
\end{eqnarray*}
where $d_1, d_2$ are constant terms. 
By (\ref{eq:vector1}) and (\ref{eq:vector2}), we also get 
\begin{eqnarray*}
{ } & { } &
H_3\circ H_2\circ H_1 \circ D_p(x_1, x_2)\\
&=&\left(
x_1^2-2\widetilde{p}_{11}x_1
+x_2^2-2\widetilde{p}_{12}x_2+d_1
, 
x_1^2-2\widetilde{p}_{21}x_1
+x_2^2-2\widetilde{p}_{22}x_2+d_2
\right), 
\\
&=&
\left(
(x_1-\widetilde{p}_{11})^2
+(x_2-\widetilde{p}_{12})^2+d_1'
, 
(x_1-\widetilde{p}_{21})^2
+(x_2-\widetilde{p}_{22})^2+d_2'
\right), 
\end{eqnarray*}
where $d_1', d_2'$ are constant terms. 

Let $H_4:\R^2\to \R^2$ be the affine transformation defined by 
\begin{eqnarray*}
H_4(X_1,X_2)=(X_1-d_1', X_2-d_2'). 
\end{eqnarray*}
Then, we have 
\begin{eqnarray*}
{ } & { } &
H_4\circ H_3\circ H_2\circ H_1 \circ D_p(x_1, x_2)\\
&=&\left(
(x_1-\widetilde{p}_{11})^2
+(x_2-\widetilde{p}_{12})^2
, 
(x_1-\widetilde{p}_{21})^2
+(x_2-\widetilde{p}_{22})^2
\right)
\\
&=&D_{\widetilde{p}}(x_1, x_2). 
\end{eqnarray*}
\end{proof}
\begin{lemma}\label{mainlem}
Let $\gamma: N\to \R^2$ be an immersion 
satisfying $(\ast)$, 
where $N$ is a manifold of dimension $1$. 
Then, 
for any non-empty open set $U_1\times U_2$ of $N\times N$, 
there exists an element $(q_1,q_2)\in U_1\times U_2$ 
such that 
\begin{eqnarray*}
\det 
\left(
\begin{array}{cccc}
\displaystyle{\frac{d \gamma_1}{dt_1}(q_1)} & 
\gamma_1(q_2)-\gamma_1(q_1) 
\\
 [4.5mm] 
\displaystyle{\frac{d \gamma_2}{dt_1}(q_1)} & 
\gamma_2(q_2)-\gamma_2(q_1) 
\end{array}
\right)\not=0, 
\end{eqnarray*}
where $\gamma=(\gamma_1, \gamma_2)$ and $t_1$ 
is a local coordinate around $q_1$. 
\end{lemma}

\begin{proof}
Let $U_1\times U_2$ be any non-empty open set of $N\times N$. 
Then, there exists a coordinate neighborhood $(U_1'\times U_2', (t_1,t_2))$ 
satisfying $U_1'\times U_2' \subset U_1\times U_2$. Fix $q_1'\in U_1'$. 

Now, suppose that for any point $t_2\in U_2'$, 
\begin{eqnarray}\label{eq:det0}
\det 
\left(
\begin{array}{cccc}
\displaystyle{\frac{d \gamma_1}{dt_1}(q_1')} & 
\gamma_1(t_2)-\gamma_1(q_1') 
\\
 [4.5mm] 
\displaystyle{\frac{d \gamma_2}{dt_1}(q_1')} & 
\gamma_2(t_2)-\gamma_2(q_1') 
\end{array}
\right)=0, 
\end{eqnarray}
where $\gamma=\left(\gamma_1, \gamma_2\right)$.
By (\ref{eq:det0}), we have 
\begin{eqnarray*}
\displaystyle{\frac{d \gamma_1}{dt_1}(q_1')}\left(\gamma_2(t_2)-\gamma_2(q_1')\right)
-
\displaystyle{\frac{d \gamma_2}{dt_1}(q_1')}\left(\gamma_1(t_2)-\gamma_1(q_1')\right)=0, 
\end{eqnarray*}
for any point $t_2\in U_2'$. 
Hence, we get 
\begin{eqnarray}
\displaystyle{\frac{d \gamma_1}{dt_1}(q_1')}
\displaystyle{\frac{d \gamma_2}{dt_2}(t_2)}
-
\displaystyle{\frac{d \gamma_2}{dt_1}(q_1')}
\displaystyle{\frac{d \gamma_1}{dt_2}(t_2)}&=&0, \label{eq:1}
\\
\displaystyle{\frac{d \gamma_1}{dt_1}(q_1')}
\displaystyle{\frac{d^2 \gamma_2}{dt_2^2}(t_2)}
-
\displaystyle{\frac{d \gamma_2}{dt_1}(q_1')}
\displaystyle{\frac{d^2 \gamma_1}{dt_2^2}(t_2)}&=&0,  \label{eq:2}
\end{eqnarray}
for any point $t_2\in U_2'$. 
By (\ref{eq:1}) and (\ref{eq:2}), we have 
\begin{eqnarray}\label{eq:3}
\left(
\begin{array}{cccc}
\displaystyle{\frac{d \gamma_2}{dt_2}(t_2)} & 
-\displaystyle{\frac{d \gamma_1}{dt_2}(t_2)} 
\\
 [4.5mm] 
\displaystyle{\frac{d ^2\gamma_2}{dt_2^2}(t_2)} & 
-\displaystyle{\frac{d ^2\gamma_1}{dt_2^2}(t_2)} 
\end{array}
\right)
\left(
\begin{array}{cc}
\displaystyle{\frac{d \gamma_1}{dt_1}(q_1')} 
\\
 [4.5mm] 
\displaystyle{\frac{d \gamma_2}{dt_1}(q_1')} 
\end{array}
\right)
=
\left(
\begin{array}{cc}
0
\\
0
\end{array}
\right), 
\end{eqnarray}
for any point $t_2\in U_2'$. 
Since $\gamma$ is an immersion, it follows that 
\begin{eqnarray}\label{eq:4}
\left(
\begin{array}{cc}
\displaystyle{\frac{d \gamma_1}{dt_1}(q_1')} 
\\
 [4.5mm] 
\displaystyle{\frac{d \gamma_2}{dt_1}(q_1')} 
\end{array}
\right)
\not=
\left(
\begin{array}{cc}
0
\\
0
\end{array}
\right). 
\end{eqnarray}
By (\ref{eq:3}) and (\ref{eq:4}), we have 
 \begin{eqnarray*}\label{eq:5}
 \det 
\left(
\begin{array}{cc}
\displaystyle{\frac{d \gamma_2}{dt_2}(t_2)} & 
-\displaystyle{\frac{d \gamma_1}{dt_2}(t_2)} 
\\
 [4.5mm] 
\displaystyle{\frac{d ^2\gamma_2}{dt_2^2}(t_2)} & 
-\displaystyle{\frac{d ^2\gamma_1}{dt_2^2}(t_2)} 
\end{array}
\right)
=0
\end{eqnarray*}
for any point $t_2\in U_2'$. 
This contradicts the hypothesis that $\gamma$ satisfies $(\ast)$. 
\end{proof}
\begin{remark}
{\rm 
It is clearly seen that 
\cref{mainlem} does not depend on the choice of a coordinate neighborhood 
containing a point $q_1$ of $N$. 
}
\end{remark}
\section{Proof of \cref{main}}\label{sec:main}
Let $O$ be any non-empty open set of $\gamma(N)\times \gamma(N)$. 
Then, there exist non-empty open sets $O_1$ and $O_2$ of $\gamma(N)$ 
satisfying $O_1\times O_2\subset O$. 
For the proof, 
it is sufficient to show that 
there exist points $p_1\in O_1$ and $p_2\in O_2$ 
such that $D_{p}\circ \gamma :N\to \R^2$ 
is an immersion with normal crossings, 
where $p=(p_1, p_2)$. 
Since $\gamma$ is continuous, 
there exist coordinate neighborhoods $(U_1, t_1)$ and $(U_2, t_2)$ of $N$ such that 
$\gamma(U_1)\subset O_1$ and $\gamma(U_2)\subset O_2$. 

Now, let $I_1$ (resp., $I_2$) be an open interval containing 
$0$ (resp., $1$) of $\R$, and let $\Phi :U_1\times U_2\times I_1\times I_2 \to \R^4$ be the mapping defined by 
\begin{eqnarray*}
\Phi(t_1,t_2,s_1,s_2)
&=&
\left(\gamma(t_1)+s_1\overrightarrow{\gamma(t_1)\gamma(t_2)}, 
\gamma(t_1)+s_2\overrightarrow{\gamma(t_1)\gamma(t_2)}\right)
\\
&=&
\left(
(1-s_1)\gamma_1(t_1)+s_1\gamma_1(t_2), 
(1-s_1)\gamma_2(t_1)+s_1\gamma_2(t_2), 
\right.
\\
&&\left.
(1-s_2)\gamma_1(t_1)+s_2\gamma_1(t_2), 
(1-s_2)\gamma_2(t_1)+s_2\gamma_2(t_2) 
\right), 
\end{eqnarray*}
where $\gamma=(\gamma_1, \gamma_2)$. 
Then, we get 
{\small 
\begin{eqnarray*}
{ } & { } &J\Phi_{(t_1,t_2,s_1,s_2)}
=
\left(
\begin{array}{cccc}
\displaystyle{(1-s_1)\frac{d \gamma_1}{dt_1}(t_1)} & 
\displaystyle{s_1\frac{d \gamma_1}{dt_2}(t_2)} &
\gamma_1(t_2)-\gamma_1(t_1) &
0 
\\
 [4.5mm] 
\displaystyle{(1-s_1)\frac{d \gamma_2}{dt_1}(t_1)} & 
\displaystyle{s_1\frac{d \gamma_2}{dt_2}(t_2)} &
\gamma_2(t_2)-\gamma_2(t_1) &
0 
\\
 [4.5mm] 
\displaystyle{(1-s_2)\frac{d \gamma_1}{dt_1}(t_1)} & 
\displaystyle{s_2\frac{d \gamma_1}{dt_2}(t_2)} &
0 &
\gamma_1(t_2)-\gamma_1(t_1) 
\\
 [4.5mm] 
\displaystyle{(1-s_2)\frac{d \gamma_2}{dt_1}(t_1)} & 
\displaystyle{s_2\frac{d \gamma_2}{dt_2}(t_2)} &
0 &
\gamma_2(t_2)-\gamma_2(t_1) 
\end{array}
\right).
\end{eqnarray*}
}
Set $s_1=0$ and $s_2=1$. Then, we have   
\begin{eqnarray*}
{ } & { } &J\Phi_{(t_1,t_2,0,1)}
=
\left(
\begin{array}{cccc}
\displaystyle{\frac{d \gamma_1}{dt_1}(t_1)} & 
0 &
\gamma_1(t_2)-\gamma_1(t_1) &
0 
\\
 [4.5mm] 
\displaystyle{\frac{d \gamma_2}{dt_1}(t_1)} & 
0 &
\gamma_2(t_2)-\gamma_2(t_1) &
0 
\\
 [4.5mm] 
0 & 
\displaystyle{\frac{d \gamma_1}{dt_2}(t_2)} &
0 &
\gamma_1(t_2)-\gamma_1(t_1) 
\\
 [4.5mm] 
 0 & 
\displaystyle{\frac{d \gamma_2}{dt_2}(t_2)} &
0 &
\gamma_2(t_2)-\gamma_2(t_1) 
\end{array}
\right). 
\end{eqnarray*}

Firstly, we will show that 
there exists an element $(\widetilde{t}_1, \widetilde{t}_2)\in U_1\times U_2$ 
such that $\det d\Phi_{(\widetilde{t}_1, \widetilde{t}_2, 0, 1)}\not =0$. 
Let $\varphi_1:U_1\times U_2\to \R$ and  $\varphi_2:U_1\times U_2\to \R$ be the functions 
defined by 
\begin{eqnarray*}
\varphi_1(t_1,t_2)=\det 
\left(
\begin{array}{cccc}
\label{eq:matrix1}
\displaystyle{\frac{d \gamma_1}{dt_1}(t_1)} & 
\gamma_1(t_2)-\gamma_1(t_1) 
\\
 [4.5mm] 
\displaystyle{\frac{d \gamma_2}{dt_1}(t_1)} & 
\gamma_2(t_2)-\gamma_2(t_1) 
\end{array}
\right), 
\\
\varphi_2(t_1,t_2)=
\det 
\left(
\begin{array}{cccc}
\displaystyle{\frac{d \gamma_1}{dt_2}(t_2)} & 
\gamma_1(t_2)-\gamma_1(t_1) 
\\
 [4.5mm] 
\displaystyle{\frac{d \gamma_2}{dt_2}(t_2)} & 
\gamma_2(t_2)-\gamma_2(t_1) 
\end{array}
\right). 
\end{eqnarray*}
Note that the function $\varphi_1$ (resp., $\varphi_2$) is 
defined by the entries of 
the $1$-th column vector and the $3$-th column vector of 
$J\Phi_{(t_1,t_2,0,1)}$ (resp., the $2$-th column vector and the $4$-th column vector of $J\Phi_{(t_1,t_2,0,1)}$). 
In order to show that there exists an element $(\widetilde{t}_1, \widetilde{t}_2)\in U_1\times U_2$ 
such that $\det d\Phi_{(\widetilde{t}_1, \widetilde{t}_2, 0, 1)}\not =0$, 
it is sufficient to show that 
there exists an element $(\widetilde{t}_1, \widetilde{t}_2)\in U_1\times U_2$ 
satisfying $\varphi_1(\widetilde{t}_1, \widetilde{t}_2)\not =0$ and 
$\varphi_2(\widetilde{t}_1, \widetilde{t}_2)\not =0$. 
From \cref{mainlem}, 
there exists $(t_1',t_2')\in U_1\times U_2$ such that $\varphi_1(t_1',t_2')\not=0$. 
Since the function $\varphi_1$ is continuous, 
there exists an open neighborhood $U_1'\times U_2'$ $(\subset U_1\times U_2)$ 
of $(t_1',t_2')$ satisfying 
$\varphi_1(t_1,t_2)\not =0$ for any $(t_1,t_2)\in U_1'\times U_2'$. 
Moreover, from \cref{mainlem}, 
there exists $(\widetilde{t}_1, \widetilde{t}_2)\in U_1'\times U_2'$ 
such that $\varphi_2(\widetilde{t}_1, \widetilde{t}_2)\not=0$. 
Namely, 
there exists an element $(\widetilde{t}_1, \widetilde{t}_2)\in U_1\times U_2$ 
such that $\det d\Phi_{(\widetilde{t}_1, \widetilde{t}_2, 0, 1)}\not =0$. 

Now, from the inverse function theorem, 
there exists an open neighborhood $V$ of 
$(\widetilde{t}_1, \widetilde{t}_2,0,1)\in U_1\times U_2\times I_1\times I_2$ 
such that $\Phi : V\to \Phi(V)$ is a diffeomorphism. 
Let $\Sigma \subset \R^2\times \R^2$ be the set 
consisting of points $p=(p_1,p_2)\in \R^4$ 
satisfying $D_p\circ \gamma :N\to \R^2$ is not 
an immersion with normal crossings. 
Note that from \cref{semimain}, 
the set $\R^4-\Sigma$ is dense in $\R^4$.
Set
\begin{eqnarray*} 
\Delta=\left \{(y_1,y_2) \in \R^2\times \R^2 \mid y_1=y_2\right \}.
\end{eqnarray*} 
Since $\Phi(V)$ is an open set of $\R^4$, 
the set $\R^4-\Sigma$ is dense in $\R^4$ and 
the set $\Delta$ is a proper algebraic set of $\R^4$, 
there exists an element $p'=(p_1',p_2')\in \Phi(V)-\Sigma \cup \Delta$. 
By $p'\not\in \Sigma$, 
the composition $D_{p'}\circ \gamma : N\to \R^2$ is an immersion with normal crossings. 
Set $(t_1', t_2', s_1', s_2')=
(\Phi |_{V})^{-1}(p_1',p_2')$. Then, we have 
\begin{eqnarray*} 
p_1'&=&\gamma(t_1')+s_1'\overrightarrow{\gamma(t_1')
\gamma(t_2')}, \\
p_2'&=&\gamma(t_1')+s_2'\overrightarrow{\gamma(t_1')
\gamma(t_2')}. 
\end{eqnarray*} 
By $p_1'\not=p_2'$, we get $\gamma(t_1')\not =\gamma(t_2')$. 
Let $L$ be the straight line defined by 
\begin{eqnarray*} 
L=\left \{\gamma(t_1')+s\overrightarrow{\gamma(t_1')
\gamma(t_2')} \ \biggr | \ s \in \R \right \}. 
\end{eqnarray*} 
Set $\widetilde{p}_1=\gamma(t_1')$ and 
$\widetilde{p}_2=\gamma(t_2')$. 
Then, it is clearly seen that $\widetilde{p}_1\in O_1$ and 
$\widetilde{p}_2\in O_2$. 
Since $p_1', p_2'\in L$ $(p_1'\not=p_2')$ and 
$\widetilde{p}_1, \widetilde{p}_2\in L$ $(\widetilde{p}_1\not=\widetilde{p}_2)$, 
from \cref{mainpro}, 
there exists an affine transformation $H:\R^2 \to \R^2$ such that
\begin{eqnarray*}
H\circ D_{p'}=D_{\widetilde{p}},
\end{eqnarray*}
where $\widetilde{p}=(\widetilde{p}_1,\widetilde{p}_2)$. 
Since $D_{p'}\circ \gamma : N\to \R^2$ is an immersion with normal crossings, 
$D_{\widetilde{p}}\circ \gamma : N\to \R^2$ is also an immersion with normal crossings.  
\hfill\qed
\section*{Acknowledgements}
The author is grateful to Takashi Nishimura 
for his kind comments. 
The author is supported by JSPS KAKENHI Grant Number 16J06911. 


\begin{thebibliography}{99}






\bibitem{CS}J.~W.~Bruce and P.~J.~Giblin, 
\textit{Curves and singularities $($second edition$)$}, 
Cambridge University Press, Cambridge, 1992. 



\bibitem{brucekirk}J.~W.~Bruce and  N.~P.~Kirk, 
\textit{Generic projections of stable mappings}, 
Bull. London Math.
Soc., \textbf{32} (2000), 718--728.




\bibitem{GG}M.~Golubitsky and V.~Guillemin, 
\textit{Stable mappings and their singularities}, 
Graduate Texts in Mathematics \textbf{14}, Springer, New York, 1973.

\bibitem{C}S.~Ichiki, 
\textit{Composing generic linearly perturbed mappings and immersions/injections}, 
to appear in J. Math.\ Soc.\ Japan 
 (available
from  arXiv:1612.01100 [math.MG]).  




\bibitem{D}S.~Ichiki and T.~Nishimura, 
\textit{Distance-squared mappings}, 
Topology Appl., 
\textbf{160} (2013), 1005--1016.   











\bibitem{GP}J. N. Mather, \textit{Generic projections}, 
Ann. of Math., (2) \textbf{98} (1973), 226--245.






  
\end{thebibliography}
\end{document}